\documentclass[proceedings,submission%
]{dmtcs}

\usepackage[latin1]{inputenc}
\usepackage{subfigure}

\usepackage[round]{natbib}
\usepackage{amssymb}
\usepackage{mathrsfs}
\usepackage{mathtools}
\usepackage{tikz,pgflibraryarrows}

\newtheorem{thm}{Theorem}
\newtheorem{lemma}[thm]{Lemma}
\newtheorem{prop}[thm]{Proposition}
\newtheorem{rem}{Remark}
\newtheorem{note}{Note}

\newtheorem{defn}{Definition}
\newtheorem{cor}{Corollary}

\author{Laurent Poinsot}
\title[Rigidity of the topological dual of spaces of formal series with respect to product topologies]{Rigidity of the topological dual of spaces of formal series with respect to product topologies}
\address{LIPN - UMR 7030, 
CNRS - Universit\'e Paris 13, F-93430 Villetaneuse,
France}
\begin{document}
\maketitle
\begin{abstract}
\paragraph{Abstract.}
Even in spaces of formal power series is required a topology in order to legitimate some operations, in particular to compute infinite summations. Many topologies can be exploited for different purposes. Combinatorists and algebraists may think to usual order topologies, or the product topology induced by a discrete coefficient field, or some inverse limit topologies. Analysists will take into account the valued field structure of real or complex numbers. As the main result of this paper we prove that the topological dual spaces of formal power series, relative to the class of product topologies with respect to Hausdorff field topologies on the coefficient field, are all the same, namely the space of polynomials. As a consequence, this kind of rigidity forces linear maps, continuous for any (and then for all) of those topologies, to be defined by very particular infinite matrices similar to row-finite matrices.       

\paragraph{R\'esum\'e.}
La validité de certaines opérations, notamment les calculs de sommes de séries, repose sur l'emploi d'une topologie sur l'espace des séries formelles. En Combinatoire ou en Algèbre il est d'usage de considérer les topologies données par une valuation, ou les topologies produits relativement à un corps discret, ou encore des topologies limites projectives. Dès lors que l'Analyse entre en scène il devient inévitable de considérer les valeurs absolues des corps des nombres réels et complexes. Il s'avère en fait que les duals topologiques des espaces de séries, par rapport aux topologies produits relatives aux copies d'un même corps topologique séparé, sont tous isomorphes à l'espace des polynômes. La preuve de ce résultat de rigidité constitue l'apport principal de notre article. L'indépendance du dual topologique relativement au choix parmi ces topologies astreint les opérateurs linéaires, continus pour une (et alors pour toutes les) topologie(s) de notre collection, à apparaître sous la forme de matrices infinies bien particulières qui n'ont qu'un nombre fini d'entrées non nulles sur chaque `` ligne ''. 

%The Spanish abstract is optional
%\paragraph{Resumen.}

\end{abstract}

%%% Please eliminate the following two lines from your file.
%\clearpage
%\tableofcontents

\section{Introduction}
\label{sec:in}

Manipulation of formal power series requires some topological notions in order to legitimate some computations. For instance, the usual substitution of a power series in one variable without constant term into an other, or the existence of the star operation, related to M\"obius inversion, are usually treated using either an order function (a valuation) or, equivalently (while more imprecise), saying that only finitely many terms contribute to the calculation in each degree. In both cases is used (explicitly or not) a topology given by a filtration (and more precisely a kind of Krull topology): the order of some partial sums must increase indefinitely for the sum to be defined (and the operation to be legal). Quite naturally other topologies may be used: for instance if $X$ is an infinite set, then the completion of the algebra $R\langle X\rangle$ of polynomials in noncommutative variables (where $R$ is a commutative ring with a unit) with respect to the usual filtration (related to the length of a word in the free monoid $X^*$) is the set of all series with only finitely many non-zero terms for each given length. The sum of the alphabet does not even exist in this completion. In order to take such series into account we must used the product topology (with a discrete $R$). If some analytical investigations must be performed (such as convergence ray, differential equation, or functional analysis), the discrete topology of $\mathbb{K}\in\{\reals,\complexes\}$ is no more sufficient; the valued field structure of $\mathbb{K}$ turns to be unavoidable. Other topologies may be used for particular needs.  

Given a topology, compatible in a certain sense with algebraic operations, on a space of formal power series with coefficients in a topological field, it can be useful to consider continuous endomorphisms since they commute to infinite sums. Quite amazingly for a very large class of possible topologies (namely product topologies with respect to separated topologies on the base field) it appears that these continuous linear maps may be seen as infinite matrices of a particular kind (each `` row '' is finitely supported) and that, independently of the topology chosen in the class. In other terms, a linear map can be represented as some `` row-finite '' matrix if, and only if, it is continuous for all the topologies in the given class (see Section~\ref{consequences}). Thus, in order to prove an endomorphism to be continuous for some topology, it suffices to prove this property for the more convenient topology in the class. Similarly, if an endomorphism is given in the form of an infinite matrix (with finitely supported `` rows ''), then it is known to be continuous for every topology in our large collection, and therefore is able to get through infinite sums. At this stage we notice that the class of topologies is that of product topologies relative to every Hausdorff topologies on the coefficient field (an infinite field has the cardinal number of the power set of its power set of distinct field topologies, see Remark~\ref{longremark}). 

The explanation of this phenomenon relies on the fact that the topological dual of a given space of formal series is the same for all the topologies in the big class: the space of polynomials. This result, presented in Section~\ref{thetheorem} and proved in Section~\ref{proof}, is our main theorem (Theorem~\ref{mainstatement}), and we present some of its direct consequences (in particular it explains the `` rigidity '' of the space of continuous endomorphisms with respect to a change of topology in Section~\ref{consequences}). 

In this paper we focus on the linear structure of formal series so that we consider any set $X$ rather than a free monoid $X^*$, and functions rather than formal series.    

\section{Some notations}

Let $R$ be a commutative ring\footnote{In the remainder, the term `` ring '' with no exception refers to `` commutative ring with a unit ''.} with a unit $1_R$. If $M$ and $N$ are two $R$-modules, then $\mathsf{Hom}_R(M,N)$ is the set of all $R$-linear maps from $M$ to $N$. In particular, the \emph{algebraic dual} $M^*$ of $M$ is the $R$-module $\mathsf{Hom}_R(M,R)$. If $X$ is a set, then $R^{(X)}$ is the $R$-module of all finitely supported maps from $X$ to $R$, that is, $f\in R^{(X)}$ if, and only if, the \emph{support} $\mathsf{supp}(f)=\{x\in X\colon f(x)\not=0\}$ of $f$ is a finite set. For every $x_0\in X$, we define the \emph{characteristic function} (or \emph{Dirac mass}) 
\begin{equation}
\begin{array}{llll}
\delta_{x_0}\colon & X & \rightarrow & R\\
&x&\mapsto&\left\{
\begin{array}{ll}
1_R & \mbox{if}\  x=x_0\ ,\\
0_R & \mbox{otherwise}\ .
\end{array}
\right .
\end{array}
\end{equation}
Let us introduce the usual \emph{evaluation map} 
\begin{equation}
\begin{array}{llll}
\mathsf{ev}\colon & R^{(X)}\otimes_R R^X&\rightarrow & R\\
& p\otimes f & \mapsto & \displaystyle\sum_{x\in X}p(x)f(x)\ .
\end{array}
\end{equation} 
As usually it is treated as a $R$-bilinear map $\langle\cdot,\cdot\rangle$, called \emph{dual pairing} (or \emph{canonical bilinear form}, see~\cite{BouAlg} or~\cite{Kothe}), \emph{i.e.}, 
$$\mathsf{ev}(p\otimes f)=\langle p,f\rangle\ .$$
In particular, for every $x\in X$, $\mathsf{ev}(\delta_x\otimes f)=\langle \delta_x,f\rangle=f(x)$, and then $\pi_x\colon f\mapsto\langle\delta_x,f\rangle$ is the \emph{projection} of $R^X$ onto $R\delta_x\cong R$. The dual pairing has the obvious properties of non-degeneracy:
\begin{enumerate}
\item for every $p\in R^{(X)}\setminus\{0\}$, there is some $f\in R^X$ such that $\langle p,f\rangle\not=0$,
\item for every $f\in R^{X}\setminus\{0\}$, there is some $p\in R^{(X)}$ such that $\langle p,f\rangle\not=0$.
\end{enumerate}

Let $E$ be a set, and $\mathcal{F}$ be a collection of maps $\phi\colon E\rightarrow E_{\phi}$, where each $E_{\phi}$ is a topological space. The \emph{initial topology} induced by $\mathcal{F}$ on $E$ is the coarsest topology that makes continuous each $\phi$. This topology is Hausdorff if, and only if, $\mathcal{F}$ \emph{separates} the elements of $E$ (\emph{i.e.}, for every $x\not=y$ in $E$, there exists some $\phi\in \mathcal{F}$ such that $\phi(x)\not=\phi(y)$), and $E_{\phi}$ is separated for each $\phi\in \mathcal{F}$. Now, if $E=F^X$, where $X$ is a set, and $F$ is a topological space, then the \emph{product topology} (or \emph{topology of simple convergence} or \emph{function topology}) on $F^X$ is the initial topology induced by the canonical projections $\pi_x\colon f\mapsto f(x)$, $x\in X$, from $E$ onto $F$. It is characterized by the following property: For every topological space $(Y,\tau_Y)$, the map $f\colon Y\rightarrow F^X$ is continuous if, and only if, every map $\pi_x\circ f\colon Y\rightarrow F$ is continuous (see~\cite{BouTop}). 

Let $R$ be a ring. It is said to be a \emph{topological ring} when it is equipped with a topology $\tau$ (not necessarily Hausdorff) for which ring operations, $x\mapsto -x$, $(x,y)\mapsto x+y$, $(x,y)\mapsto xy$, are continuous when is considered on $R\times R$ the product topology defined by $\tau$ on each factors. A field $\mathbb{K}$ which is also a topological ring is said to be a \emph{topological field} when the map $x\mapsto x^{-1}$ is continuous on $\mathbb{K}^*=\mathbb{K}\setminus\{0\}$ (equipped with the subspace topology). 

If $R$ (resp. $\mathbb{K}$) is a topological ring (resp. topological field),  any $R$-module $M$ (resp. $\mathbb{K}$-vector space) is said to be a \emph{topological $R$-module} (resp. a \emph{topological $\mathbb{K}$-vector space}) if it is equipped with a topology (Hausdorff or not) that makes continuous the module maps $x\mapsto -x$, $(x,y)\mapsto x+y$, $(\lambda, x)\mapsto \lambda x$. Notice that if $R$ (resp. $\mathbb{K}$) is a topological ring (resp. topological field), and $X$ is any set, then $R^X$ (resp. $\mathbb{K}^X$) is a topological $R$-module (resp. topological $\mathbb{K}$-vector space)  -- in the obvious way -- when endowed with the product topology. 

Let $R$ be a topological ring, and $M$ be a topological $R$-module. Then $M^{\prime}$ denotes the \emph{topological dual} of $M$, that is, the sub-module of $M^*$ of continuous linear forms\footnote{Quite obviously, $M^*$ is the topological dual of $M$, when $R$ and $M$ both have the discrete topology.}.  

\paragraph{Summability.}

Many intermediary results of this papers require the notion, and some properties, of a summable family in a topological ring (which is actually hidden but fundamental for the treatment of formal power series). We recall them without any proof; we freely used them in the sequel, and  refer to~\cite{Warner} for further information concerning this concept but also topological rings, fields and modules. 
\begin{defn}
Let $G$ be a Hausdorff Abelian group (that is, an Abelian group -- in additive notation -- with a separated topology such that the group operations, addition and inversion, are continuous), and $(x_i)_{i\in I}$ be a family of elements of $G$. An element $s\in G$ is the \emph{sum} of the \emph{summable family} $(x_i)_{i\in I}$ if, and only if, for each neighborhood $V$ of $s$ there exists a finite subset $J\subseteq I$ such that $\displaystyle\sum_{j\in J}x_j\in V$.  
\end{defn}
The sum $s$ of a summable family $(x_i)_{i\in I}$ of elements of $G$ is usually denoted by $\displaystyle\sum_{i\in I}x_i$.
\begin{prop}
If $(x_i)_{i\in I}$ is a summable family of elements of a Hausdorff Abelian group $G$ having a sum $s$, then for any permutation $\sigma$ of $I$, $s$ is also the sum of the summable family $(x_{\sigma(i)})_{i\in I}$.
\end{prop}
\begin{prop}
If $(x_i)_{i\in I}$ is a summable family of elements of a Hausdorff Abelian group $G$, then for every neighborhood $V$ of zero, $x_i\in V$ for all but finitely many $i\in I$.
\end{prop}
\begin{prop}
Let $G$ be the Cartesian product of a family $(G_{\lambda})_{\lambda\in L}$ of Hausdorff Abelian groups ($G$ has the product topology). Then $s$ is the sum of a family $(x_i)_{i\in I}$ of elements of $G$ if, and only if, $\pi_{\lambda}(s)$ is the sum of 
$(\pi_{\lambda}(x_i))_{i\in I}$ for each $\lambda\in L$ (where $\pi_{\lambda}$ is the canonical component from $G$ onto $G_{\lambda}$). 
\end{prop}
\begin{prop}
If $\phi$ is a continuous homomorphism from a Hausdorff Abelian group $G_1$ to a Hausdorff Abelian group $G_2$, and if $(x_i)_{i\in I}$ is a summable family of elements of $G_1$, then $(\phi(x_i))_{i\in I}$ is summable in $G_2$, and $\displaystyle\sum_{i\in I}\phi(x_i)=\phi\left( \sum_{i\in I}x_i\right )$.
\end{prop}

\section{The statement}\label{thetheorem}

The objective of this paper is to prove the following result, and to observe some its consequences. 
\begin{thm}\label{mainstatement}
Let ${\mathbb{K}}$ be an Hausdorff topological field (it might be non-discrete!), and $X$ be any set. Then, the topological dual $({\mathbb{K}}^X)^{\prime}$ of ${\mathbb{K}}^X$, under the product topology, is isomorphic (as a ${\mathbb{K}}$-vector space) to ${\mathbb{K}}^{(X)}$.
\end{thm}

In order to illustrate the scope of this result, let us assume for an instant that $\mathbb{K}$ is a discrete field, and that $V$ is an Hausdorff topological $\mathbb{K}$-vector space. The topology on $V$ is said to be \emph{linear} if, and only if, it has a neighborhood basis of zero consisting of vector subspaces open in the topology.
\begin{prop}[\cite{Lefschetz,Dieudonne}]
Let $V$ be a $\mathbb{K}$-vector space with a linear topology. Then the following conditions are equivalent:
\begin{enumerate}
\item $V$ is complete, and all its open subspaces are of finite codimension.
\item $V$ is an inverse limit of discrete finite-dimensional vector spaces, with the inverse limit topology.
\item $V$ is isomorphic, as a topological vector space, to the algebraic dual $W^{*}$ of a discrete vector space $W$, with the topology of simple convergence. Equivalently, $V$ is isomorphic to $\mathbb{K}^X$, with the product topology, for some set $X$. 
\end{enumerate}
\end{prop} 
A topological vector space with the above equivalent properties is called \emph{linearly compact}. This kind of spaces is not so important for this paper, but we obtain a characterization of their topological duals as a minor consequence of our main result: we get that the topological dual of any linearly compact vector space is isomorphic to some $\mathbb{K}^{(X)}$. 

We can also deduce an immediate corollary that is a partial reciprocal of our main result.
\begin{cor}\label{cor1}
Let $\mathbb{K}$ be a topological field, and $X$ be a set. Let us assume that $\mathbb{K}^X$ has the product topology. Suppose that $X$ is infinite. Then,  
$(\mathbb{K}^{X})^{\prime}$ is isomorphic to $\mathbb{K}^{(X)}$ if, and only if, $\mathbb{K}$ is Hausdorff.
\end{cor}
\begin{proof}
If $\mathbb{K}$ is Hausdorff, then according to Theorem~\ref{mainstatement} $(\mathbb{K}^{X})^{\prime}$ is isomorphic to $\mathbb{K}^{(X)}$. Now, let $R$ be an indiscrete topological ring. Then, with respect to the trivial topology on $R$, $(R^{X})^{\prime}=(R^{X})^{*}$. Because $X$ is infinite, the field $\mathbb{K}$ cannot be indiscrete. But ring topologies on a field (and in particular field topologies) may be either Hausdorff or the indiscrete one (see~\cite{Warner}).   
\end{proof}

\section{The proof of Theorem~\ref{mainstatement}}\label{proof}

\begin{lemma}\label{lem1}
Let $R$ be a commutative ring with unit, and $X$ be a set. 
Let us define 
\begin{equation}
\begin{array}{llll}
\Phi\colon& R^{(X)} & \rightarrow & R^{R^X}\\
& p & \mapsto & \left ( 
\begin{array}{llll}
\Phi(p)\colon&R^X & \rightarrow & R\\
&f & \mapsto & \langle p,f\rangle
\end{array}\right )\ .
\end{array}
\end{equation}
Then, for every $p\in R^{(X)}$, $\Phi(p)\in (R^X)^*$, $\Phi$ is $R$-linear and one-to-one.
\end{lemma}
\begin{proof}
The first and second properties are obvious. Let $p\in R^{(X)}$ such that $\Phi(p)=0$, then for every $f\in R^{X}$, 
$\Phi(p)(f)=0$, and in particular, for every $x\in X$, $0=\Phi(p)(\delta_x)=\langle p,\delta_x\rangle=p(x)$, in such a way that $p=0$. 
\end{proof}
\begin{lemma}\label{lem2}
Assume that $R$ is a topological ring (Hausdorff or not) and that $R^X$ has the product topology. Then for every $p\in R^{(X)}$, $\Phi(p)$ is continuous, \emph{i.e.}, $\Phi(p)\in (R^X)^{\prime}$. 
\end{lemma}
\begin{proof}
It is clear since $\Phi(p)$ is a finite sum of (scalar multiples) of projections. 
\end{proof}
\begin{lemma}\label{lem3}
Let us suppose that $R$ is an Hausdorff topological ring, $X$ is a set, and that $R^X$ has the product topology. For every $f\in R^X$, the family $(f(x)\delta_x)_{x\in X}$ is summable with sum $f$.
\end{lemma}
\begin{proof}
It is sufficient to prove that for every $x_0\in X$ the family $$(\langle\delta_{x_0},f(x)\delta_x\rangle)_{x\in X}=(f(x)\delta_x(x_0))_{x\in X}$$ is summable in $R$ with sum $\langle \delta_{x_0},f\rangle=f(x_0)$ which is immediate.
\end{proof}
\begin{lemma}\label{lem4}
Under the same assumptions as Lemma~\ref{lem3}, if $\ell\in (R^{X})^{\prime}$, then $$Y_{\ell}=\{x\in X\colon \ell(\delta_x)\ \mbox{is invertible in}\ R\}$$ is finite. 
\end{lemma}
\begin{proof}
Since $\ell$ is continuous (and linear), and $(f(x)\delta_x)_{x\in X}$ is summable with sum $f$, then $(f(x)\ell(\delta_x))_{x\in X}$ is summable in $R$ with sum $\ell(f)$ for every $f\in R^X$. Let us define 
\begin{equation}
\begin{array}{llll}
f\colon & X&\rightarrow & R\\
& x&\mapsto & \left \{ 
\begin{array}{ll}
\ell(\delta_x)^{-1} & \mbox{if}\ x\in Y_{\ell}\ ,\\
0_R& \mbox{otherwise}\ .
\end{array}\right .
\end{array}
\end{equation}
Then, $(f(x)\ell(\delta_x))_{x\in X}$ is summable with sum $\ell(f)$ in $R$. According to properties of summability, for every neighborhood $U$ of $0_R$ in $R$, $f(x)\ell(\delta_x)\in U$ for all but finitely many $x\in X$. Since $R$ is assumed Hausdorff, there is some neighborhood $U$ of $0_R$ such that $1_R\not\in U$. If $Y$ is not finite, then $1_R=f(x)\ell(\delta_x)\not\in U$ for every $x\in Y$ which is a contradiction.
\end{proof}
\begin{lemma}\label{lem5}
Let ${\mathbb{K}}$ be an Hausdorff topological field, $X$ be a set, and assume that ${\mathbb{K}}^X$ is equipped with the product topology. Let $\ell\in ({\mathbb{K}}^X)^*$. If $\ell\in ({\mathbb{K}}^X)^{\prime}$, then $\ell(\delta_x)=0$ for all but finitely many $x\in X$. 
\end{lemma}
\begin{proof}
According to Lemma~\ref{lem4}, the set $\{x\in X\colon \ell(\delta_x)\ \mbox{is invertible in}\ {\mathbb{K}}\}=\{x\in X\colon \ell(\delta_x)\not=0\}$ is finite. 
\end{proof}
\begin{lemma}\label{lem6}
Under the same assumptions as Lemma~\ref{lem5}, $\Phi$ is onto. 
\end{lemma}
\begin{proof}
Let $\ell\in ({\mathbb{K}}^X)^{\prime}$. Let us define 
\begin{equation}
\begin{array}{llll}
p_{\ell}\colon & X& \rightarrow & {\mathbb{K}}\\
& x& \mapsto & \ell(\delta_x)\ .
\end{array}
\end{equation}
\emph{A priori} $p_{\ell}\in {\mathbb{K}}^{X}$. But according to Lemma~\ref{lem5}, actually $p_{\ell}\in {\mathbb{K}}^{(X)}$. Let $f\in {\mathbb{K}}^{X}$. We have $$\Phi(p_{\ell})(f)=\langle p_{\ell},f\rangle=\displaystyle\sum_{x\in X}p_{\ell}(x)f(x)=\sum_{x\in X}\ell(\delta_x)f(x)=\ell(f)$$ 
and then $\Phi(p_{\ell})=\ell$. 
\end{proof}

Now it is easy to conclude the proof of Theorem~\ref{mainstatement} since it follows directly from Lemmas~\ref{lem1} and~\ref{lem6}. 
\begin{rem}\label{longremark}
\begin{enumerate}
\item The algebraic dual of ${\mathbb{K}}^X$ may be distinct from ${\mathbb{K}}^{(X)}$. Indeed, let $(e_i)_{i\in I}$ be an algebraic basis of ${\mathbb{K}}^{X}$ (the existence of such a basis requires the axiom of choice for sets $X$ of arbitrary large cardinal number). Therefore, every map $f\in {\mathbb{K}}^{X}$ may be (uniquely) written as a finite linear combination $\displaystyle\sum_{i\in I}f_i e_i$, with $f_i\in {\mathbb{K}}$ for each $i\in I$. Let us consider the map $\ell\colon {\mathbb{K}}^X\rightarrow{\mathbb{K}}$ such that $\ell(f)=\displaystyle\sum_{i\in I}f_i$. Clearly, $\ell$ is a linear form, that is, an element of the algebraic dual $({\mathbb{K}}^X)^*$ of ${\mathbb{K}}^X$. The family $(\delta_x)_{x\in X}$ is linearly independent in ${\mathbb{K}}^X$. Thus we may consider the algebraic basis of ${\mathbb{K}}^X$ that extends $(\delta_x)_{x\in X}$. Now, the corresponding functional $\ell$ has a nonzero value for each $\delta_x$. Therefore, if $X$ is infinite, then $\ell$ does not belong to the image of $\Phi$, or, in other terms, $\ell\not\in ({\mathbb{K}}^X)^{\prime}$. In particular, whenever ${\mathbb{K}}$ is an Hausdorff topological field, ${\mathbb{K}}^X$ has the product topology, and $X$ is infinite, then $\ell$ is  discontinuous at zero (and thus on the whole ${\mathbb{K}}^X$). 
\item A field topology may be either Hausdorff or the indiscrete one. Its seems to remain too little choice. Actually it is known (see~\cite{Podewski,Kiltinen}) that every infinite field $\mathbb{K}$ has $2^{2^{|\mathbb{K}|}}$ non-homeomorphic topologies (where $|X|$ denotes the cardinal number of a set $X$). 
\item Let $R$ be a discrete ring, and $X$ be a set. Let $X^*$ be the free monoid on the alphabet $X$, $\epsilon$ be the empty word, and $|\omega|$ be the length of a word $\omega\in X^*$. Let us define $\mathfrak{M}_{\geq n}=\{f\in R^X\colon \nu(f)\geq n\}$, $n\in\naturals$, where $\nu(f)=\inf\{n\in\naturals\colon \exists \omega\in X^*,\ |\omega|\not=n,\ \mbox{and}\ f(\omega)\not=0\}$ for every non-zero $f\in R^X$ (the infimum being taken on $\naturals\cup\{\infty\}$, with $\infty>n$ for every $n\in \naturals$, then $\nu(0)=\infty$). The set $R^X$, seen as the $R$-algebra $R\langle\langle X\rangle\rangle$ of formal power series in noncommutative variables, may be topologized (as a topological $R$-algebra -- see~\cite{Warner} -- and, therefore, as a topological $R$-module) by the decreasing filtration of ideals $\mathfrak{M}_{\geq n}$: this is an example of the so-called \emph{Krull topology} (see~\cite{Eisenbud}), and it is the usual topology considered for formal power series in combinatorics and algebra; in case $X$ is reduced to an element $x$, we recover the usual $\mathfrak{M}$-adic topology of $\mathbb{K}[[x]]$, where $\mathfrak{M}=\langle x\rangle$ is the principal ideal generated by $x$. Whenever $X$ is finite, this Krull topology coincides\footnote{Notice that if $X$ is infinite both topologies are distinct: for instance, let $(x_n)_{n\geq 0}$ be a sequence of distinct elements of $X$, then this family is easily seen summable in the product topology with $R$ discrete, while it does not converge in the Krull topology.} with the product topology with a discrete $R$. According to Theorem~\ref{mainstatement}, in case where $X$ is finite and $R$ is a topological field $\mathbb{K}$, the topological dual of $\mathbb{K}\langle\langle X\rangle\rangle$ (which is also a linearly compact space) is the space of polynomials $\mathbb{K}\langle X\rangle$ in noncommutative variables.
\item Take any monoid with a zero (see~\cite{Clifford}) with the finite decomposition property (see~\cite{BouAlg,Poinsot}), and let $R$ be a ring. Let us consider the total contracted $R$-algebra  $R_0[[M]]$ of the monoid with zero $M$ (see~\cite{Poinsot}) that consists, as a $R$-module, of $\{f\in R^M\colon f(0_M)=0_R\}$, where $0_M$ is the zero of $M$, while $0_R$ is the zero of the ring $R$. It is clear that $R_0[[M]]\cong R^{M^{*}}$ (as $R$-algebras), with $M^*=M\setminus\{0_M\}$. Now, let us assume that $\mathbb{K}$ is an Hausdorff topological field. The product topology on $\mathbb{K}^M$ induces the product topology on $\mathbb{K}^{M^*}$, that corresponds to that of $\mathbb{K}_0[[M]]$. The topological dual of $\mathbb{K}^{M^*}$ being $\mathbb{K}^{(M^*)}$ it is easy to check that $(\mathbb{K}_0[[M]])^{\prime}$ is isomorphic to the (usual) contracted algebra (see~\cite{Clifford}) $\mathbb{K}_0[M]$ of the monoid with zero $M$.  
\item Theorem~\ref{mainstatement} may be applied for the discrete topology on ${\mathbb{K}}\in\{\reals,\complexes\}$, but also for the usual topologies of $\reals$ and $\complexes$, in such a way that the topological dual spaces $(\reals^{X})^{\prime}$ or $(\complexes^{X})^{\prime}$ for both discrete topology and the topology induced by the (usual) absolute values on $\reals,\complexes$ are identical since isomorphic to $\reals^{(X)}$ or $\complexes^{(X)}$. Notice that $\mathbb{K}^X$ is a Fr\'echet space (see~\cite{Treves}), real or complex depending on whether $\mathbb{K}=\mathbb{R}$ or $\mathbb{K}=\mathbb{C}$, when is considered the product topology relative to the absolute value, and as such allows functional analysis like, for instance, Banach-Steinhaus, open map and closed graph theorems that do not hold in the case of the same space with the product topology relative to a discrete $\mathbb{K}$.   
\end{enumerate}
\end{rem}

\section{Consequence on continuous endomorphisms}\label{consequences}

As explained in the Introduction, the rigidity of the dual space with respect to the change of product topologies forces continuous linear maps (with respect to any of those topologies) to be represented by `` row-finite '' matrices. 

Let ${\mathbb{K}}$ be an Hausdorff topological field, and $X$, $Y$ be two sets. We suppose that ${\mathbb{K}}^{Z}$ has the product topology for $Z\in\{X,Y\}$. The set of all linear maps (resp. continuous linear maps) from ${\mathbb{K}}^X$ to ${\mathbb{K}}^Y$ is denoted by $\mathsf{Hom}_{{\mathbb{K}}}({\mathbb{K}}^X,{\mathbb{K}}^Y)$ (resp. $\mathcal{L}({\mathbb{K}}^X,{\mathbb{K}}^Y)$).  We denote by ${\mathbb{K}}^{Y\times(X)}$ the set of all maps $M\colon Y\times X\rightarrow {\mathbb{K}}$ such that for each $y\in Y$, the set $\{x\in X\colon M(y,x)\not=0\}$ is finite. Recall that if $p\in {\mathbb{K}}^{(X)}$, then its support is given by  
\begin{equation}
\mathsf{supp}(p)=\{x\in X\colon p(x)\not=0\}\ .
\end{equation}

Let $\phi\in\mathcal{L}({\mathbb{K}}^X,{\mathbb{K}}^Y)$. We define the following map:
\begin{equation}
\begin{array}{llll}
M_{\phi}\colon & Y\times X&\rightarrow & {\mathbb{K}}\\
& (y,x) & \mapsto & \langle \delta_y,\phi(\delta_x)\rangle\ .
\end{array}
\end{equation}
\begin{lemma}\label{lemsecondepartie}
For each $\phi\in\mathcal{L}({\mathbb{K}}^X,{\mathbb{K}}^Y)$, $M_{\phi}\in {\mathbb{K}}^{Y\times(X)}$, and the map $\phi\mapsto M_{\phi}$ is into.
\end{lemma}
\begin{proof}
For every $x\in X$, the map 
\begin{equation}
\begin{array}{lll}
{\mathbb{K}}^X&\rightarrow & {\mathbb{K}}\\
f&\mapsto & \langle \delta_y,\phi(f)\rangle
\end{array}
\end{equation}
is an element of $({\mathbb{K}}^X)^{\prime}$ (because it is the composition of $\phi$ and the projection onto ${\mathbb{K}}\delta_y$). According to Theorem~\ref{mainstatement}, there is one and only one $p_{\phi,y}\in {\mathbb{K}}^{(X)}$ such that for every $f\in {\mathbb{K}}^X$, 
$\langle p_{\phi,y},f\rangle=\langle \delta_y,\phi(f)\rangle$.  In particular, for every $x\in X$, 
\begin{equation}
\displaystyle\left\{\begin{array}{ll} p_{\phi,y}(x)&\mbox{if}\ x\in \mathsf{supp}(p_{\phi,y})\\
0_R & \mbox{otherwise}\end{array}\right .=\sum_{z\in X}p_{\phi,y}(z)\delta_x(z)=\langle p_{\phi,y},\delta_x\rangle=\langle \delta_y,\phi(\delta_x)\rangle\ .
\end{equation}
Therefore $\{x\in X\mid \langle\delta_y,\phi(\delta_x)\rangle\not=0\}=\mathsf{supp}(p_{\phi,y})$, and then $M_{\phi}\in {\mathbb{K}}^{Y\times(X)}$. 

Suppose that $M_{\phi}=M_{\phi^{\prime}}$, then for every $(y,x)\in Y\times X$, $\langle \delta_y,\phi(\delta_x)\rangle=\langle 
\delta_y,\phi^{\prime}(\delta_x)\rangle$. Then, by bilinearity, $\phi(\delta_x)(y)-\phi^{\prime}(\delta_x)(y)=\langle \delta_y,\phi(\delta_x)-\phi^{\prime}(\delta_x)\rangle=0$. Since this last equality holds for every $y \in Y$, $\phi(\delta_x)=\phi^{\prime}(\delta_x)$ for every $x\in X$. Now, let $f\in {\mathbb{K}}^X$, since $f=\displaystyle\sum_{x\in X}f(x)\delta_x$ (sum of a summable family), by continuity, $\phi(f)=\displaystyle\sum_{x\in X}f(x)\phi(\delta_x)=\sum_{x\in X}f(x)\phi^{\prime}(\delta_x)=\phi^{\prime}(f)$.   
\end{proof}

\begin{thm}\label{infinitematrices}
The sets $\mathcal{L}({\mathbb{K}}^X,{\mathbb{K}}^Y)$ and ${\mathbb{K}}^{Y\times (X)}$ are equipotent. More precisely the map $\phi\mapsto M_{\phi}$ of Lemma~\ref{lemsecondepartie} is onto.
\end{thm}
\begin{proof}
Let $M\in {\mathbb{K}}^{Y\times(X)}$. Let us define $\psi_M\colon {\mathbb{K}}^X\rightarrow {\mathbb{K}}^Y$ by $$\displaystyle\psi_M(f)=\psi_M(\sum_{x\in X}f(x)\delta_x)=
\sum_{y\in Y}\left ( \sum_{x\in X}M(y,x)f(x)\right )\delta_y\ .$$ (Clearly the second sum on $x\in X$ has only finitely many non zero terms since $M\in {\mathbb{K}}^{Y\times(X)}$, and therefore is defined in ${\mathbb{K}}$.) The map $\psi_M$ is linear. To see this, since $\psi_M(f)\in {\mathbb{K}}^X$, it is sufficient to prove that for every $\lambda\in {\mathbb{K}}$, $f,g\in {\mathbb{K}}^{X}$, and every $y \in Y$, 
$\psi_M(\lambda f+g)(y)=\lambda\psi_M(f)(y)+\psi_M(g)(y)$. This equality to prove is equivalent to the following: 
\begin{equation}
\begin{array}{llll}
&\psi_M(\lambda f+g)(y)&=&\lambda\psi_M(f)(y)+\psi_M(g)(y)\\
\Leftrightarrow & \langle \delta_y,\psi_M(\lambda f+g)\rangle&=&\lambda\langle \delta_y,\psi_M(f)\rangle+\langle\delta_y,\psi_M(g)\rangle\\
\Leftrightarrow & \langle \delta_y,\psi_M(\lambda f+g)\rangle&=&\langle\delta_y,\lambda \psi_M(f)+\psi_M(g)\rangle\\
\Leftrightarrow  & \displaystyle\sum_{x\in X}M(y,x)(\lambda f(x)+g(x))&=&\displaystyle\sum_{x\in X}(\lambda M(y,x)f(x)+M(y,x)g(x))\ .
\end{array}
\end{equation}
But the last equality is obvious (since the sums have only finitely many non zero terms). Let us prove that $\psi_M$ is continuous. It is sufficient to prove that for every $y \in Y$, $\ell_{M,y}\colon {\mathbb{K}}^X\rightarrow {\mathbb{K}}$, defined by 
$\displaystyle\ell_{M,y}(f)=\langle \delta_y,\psi_M(f)\rangle=\sum_{x\in X}M(y,x)f(x)$, is continuous. This is the case since 
$\ell_{M,y}$ is a finite sum of scalar multiples of projections. Therefore $\psi_M\in \mathcal{L}({\mathbb{K}}^X)$. Finally we prove that $M_{\psi_M}=M$. Let $(y,x)\in Y\times X$, we have 
\begin{equation}
\begin{array}{lll}
M_{\psi_M}(y,x)&=&\langle\delta_y,\psi_M(\delta_x)\rangle\\
&=&\displaystyle\langle\delta_y,\sum_{y^{\prime}\in Y}\left (\sum_{z\in X}M(y^{\prime},z)\delta_x(z)\right)\delta_{y^{\prime}}\rangle\\
&=&\displaystyle\sum_{z\in X}M(y,z)\delta_x(z)\\
&=&M(y,x)\ .
\end{array}
\end{equation}
The map $\phi\mapsto M_{\phi}$ is thus onto, and, by Lemma~\ref{lemsecondepartie}, it is a bijection.
\end{proof}
\begin{note}
\begin{enumerate}
\item If $X=Y$, then denoting the set $\mathcal{L}({\mathbb{K}}^X,{\mathbb{K}}^X)$ of all continuous endomorphisms by $\mathcal{L}({\mathbb{K}}^X)$, we find that $\mathcal{L}({\mathbb{K}}^X)$ and ${\mathbb{K}}^{X\times (X)}$ are equipotent.
\item If $X=Y=\naturals$, then is recovered the usual notion of row-finite matrices (see~\cite{Cooke}).
\item If $Y$ is reduced to a single element $x$, then $\mathcal{L}({\mathbb{K}}^X)=\mathcal{L}({\mathbb{K}}^X,{\mathbb{K}})\cong \mathcal{L}({\mathbb{K}}^X,{\mathbb{K}}^{\{x\}})\cong {\mathbb{K}}^{\{x\}\times (X)}\cong {\mathbb{K}}^{(X)}$, which is our Theorem~\ref{mainstatement}.  
\end{enumerate}
\end{note}

\section{Topological duality and completion}

In this section, we construct explicitly the canonical isomorphism between the topological duals of $\mathbb{K}^{(X)}$ and $\mathbb{K}^X$ using the fact that the former is the completion of the later.

Recall the following definition: let ${\mathbb{K}}$ be a field with an Hausdorff (field) topology, and $V$ be an Hausdorff topological ${\mathbb{K}}$-vector space. The \emph{completion} of $V$ is a pair $(\widehat{V},i)$ where $\widehat{V}$ is complete (Hausdorff) topological ${\mathbb{K}}$-vector space and $i\colon V\rightarrow \widehat{V}$ such that 
\begin{enumerate}
\item The map $i$ is an isomorphism of topological ${\mathbb{K}}$-vector space structures from $V$ into $\widehat{V}$, \emph{i.e.}, $i$ is both an (algebraic) isomorphism and a homeomorphism into $\widehat{V}$ (or in other terms, $i$ is a bicontinuous one-to-one $R$-linear map, or $i$ is a continuous one-to-one $R$-linear map and its inverse $i^{-1}\colon i(V)\rightarrow V$ is continuous where $i(V)$ has the subspace topology induced by $\widehat{V}$);
\item The image $i(V)$ is dense in $\widehat{V}$;
\item For any complete (Hausdorff) ${\mathbb{K}}$-vector space $W$ and any continuous and linear map $\phi\colon V \rightarrow W$, there exists one and only one map $\widehat{\phi}\colon \widehat{V}\rightarrow W$ such that $\widehat{\phi}\circ i=\phi$.   
\end{enumerate}

\begin{rem}
According to~\cite{BouTop}, $\mathbb{K}^X$ is complete and separated (with respect to the product topology) if, and only if, $\mathbb{K}$ is itself a complete Hausdorff field. In particular, the topological dual of $\mathbb{K}^X$ does not depend on whether or not $\mathbb{K}$ is a complete field.  
\end{rem}

Now let us assume that ${\mathbb{K}}$ has a the discrete topology (therefore ${\mathbb{K}}$ is a complete Hausdorff topological field). Clearly, ${\mathbb{K}}^X$ is the completion of ${\mathbb{K}}^{(X)}$ (the later being equipped with the initial topology with respect to the obvious projections which coincides with the subspace topology induced by ${\mathbb{K}}^X$ for its product topology). According to the definition of a completion, taking ${\mathbb{K}}$ in place of $W$, it is clear that there exists a canonical isomorphism $\Psi$ between the topological dual of ${\mathbb{K}}^{(X)}$ and ${\mathbb{K}}^X$. The isomorphism 
\begin{equation}
\begin{array}{llll}
\Psi\colon & ({\mathbb{K}}^{(X)})^{\prime}& \rightarrow & ({\mathbb{K}}^X)^{\prime}\\
& \ell & \mapsto & \widehat{\ell}
\end{array}
\end{equation} 
has inverse $\Psi^{-1}(\ell)=\ell\circ i=\ell|_{{\mathbb{K}}^{(X)}}$ for $\ell\in ({\mathbb{K}}^X)^{\prime}$. The isomorphism $\Psi$ may be given a precise definition. Let $\ell\in ({\mathbb{K}}^{(X)})^{\prime}$. Then we have
\begin{equation}
\begin{array}{llll}
\Psi(\ell)=\widehat{\ell}\colon & {\mathbb{K}}^X& \rightarrow & {\mathbb{K}}\\
& f & \mapsto & \displaystyle\sum_{x\in X}f(x)\ell(\delta_x)\ .
\end{array}
\end{equation}
Indeed since $f=\displaystyle\sum_{x\in X}f(x)\delta_x$ (sum of a summable family), we have 
\begin{equation}
\begin{array}{lll}
\Psi(\ell)(f)&=&\widehat{\ell}(f)\\
&=&\displaystyle \widehat{\ell}(\sum_{x\in X}f(x)\delta_x)\\
&=&\displaystyle\sum_{x\in X}f(x)\widehat{\ell}(\delta_x)\\
&&\mbox{(since $\widehat{\ell}$ is continuous)}\\
&=&\displaystyle\sum_{x\in X}f(x)\ell(\delta_x)\ .\\
&&\mbox{(since $\delta_x\in {\mathbb{K}}^{(X)}$)}
\end{array}
\end{equation}
According to previously introduced notations, $\Psi(\ell)(f)=\widehat{\ell}(f)=\langle p_{\Psi(\ell)},f\rangle$, where we recall that  $p_{\Psi(\ell)}\in{\mathbb{K}}^{(X)}$ is defined by $p_{\Psi(\ell)}(x)=\Psi(\ell)(\delta_x)=\widehat{\ell}(\delta_x)=\ell(\delta_x)$ (since $\delta_x\in {\mathbb{K}}^{(X)}$). Note that $p_{\Psi(\ell)}=\Phi^{-1}(\Psi(\ell))=\Phi^{-1}(\widehat{\ell})$. The map $\Phi^{-1}\circ \Psi\colon ({\mathbb{K}}^{(X)})^{\prime}\rightarrow {\mathbb{K}}^{(X)}$ is an isomorphism (composition of isomorphisms). The polynomial $\Phi^{-1}(\Psi(\ell))\in {\mathbb{K}}^{(X)}$ for $\ell\in ({\mathbb{K}}^{(X)})^{\prime}$ is thus given by $\Phi^{-1}(\Psi(\ell))(x)=\ell(\delta_x)$ for $x\in X$. We can check that $\ell(p)=\langle \Phi^{-1}(\Psi(\ell)),i(p)\rangle$ for any $\ell\in ({\mathbb{K}}^{(X)})^{\prime}$, and $p\in {\mathbb{K}}^{(X)}$. Indeed, $\ell(p)=\Psi(\ell)(i(p))=\langle \Phi^{-1}(\Psi(\ell)),i(p)\rangle$. 

\section{Weak topology}

Let $\mathbb{K}$ be a topological field, $(V,\tau)$ be a topological $\mathbb{K}$-vector space, and $V^{\prime}$ be its topological dual. We call \emph{$V^{\prime}$-weak topology} the weakest topology on $V$ such that the elements of $V^{\prime}$ are continuous. Let us denote by $\tau_w$ this topology. Since the elements of $V^{\prime}$ are continuous (for $\tau$), we have $\tau_{w}\subseteq \tau$. It can be shown that $(V,\tau_w)$ is also a $\mathbb{K}$-topological vector space, separated if $\mathbb{K}$ is so and $V^{\prime}$ separates the elements of $V$. The topological dual $V_w^{\prime}$ of $(V,\tau_w)$ is called \emph{weak dual} of $V$. 

\begin{cor}
Let $(\mathbb{K},\tau)$ be an Hausdorff topological field, $X$ be a set, and $\mathbb{K}^X$ endowed with the product topology, denoted by $\pi(X,\tau)$, of $|X|$ copies of $(\mathbb{K},\tau)$. The weak dual $(\mathbb{K}^X)^{\prime}_w$ of $(\mathbb{K}^X,\tau_w)$ is (isomorphic to) $\mathbb{K}^{(X)}$.  
\end{cor}

\begin{proof}
An element of $\mathbb{K}^{(X)}$ is obviously a continuous linear form (\emph{via} the isomorphism between $(\mathbb{K}^X)^{\prime}$ and $\mathbb{K}^{(X)}$ of Theorem~\ref{mainstatement}) by definition of the $\mathbb{K}^{(X)}$-weak topology. Therefore, $\mathbb{K}^{(X)}\subseteq (\mathbb{K}^{X})^{\prime}_w$. Let us prove that the product and the weak topologies are actually equal. The weak topology $\tau_w$ is, by definition, the weakest topology for which every linear form $\langle p,\cdot\rangle\colon\mathbb{K}^X\rightarrow (\mathbb{K},\tau)$ is continuous ($p\in\mathbb{K}^{(X)}$). In particular, the projections $\langle\delta_x,\cdot\rangle$ are also continuous. Therefore $\tau_w$ is stronger than $\pi(X,\tau)$ (since the later is the weakest topology with this property). So $\pi(X,\tau)\subseteq \tau_w$. Let us prove the reciprocal inclusion. To do so, it is sufficient to prove that the identity map $\mathsf{id}\colon (\mathbb{K}^X,\pi(X,\tau))\rightarrow(\mathbb{K}^X,\tau_w)$ is continuous, which is equivalent (according to usual properties of initial topology) to the fact that  for every $p\in \mathbb{K}^{(X)}$, the map 
\begin{equation}
\begin{array}{llll}
\langle p,\cdot\rangle \colon&(\mathbb{K}^X,\pi(X,\tau))&\rightarrow &(\mathbb{K},\tau)\\
& f & \mapsto & \langle p,f\rangle 
\end{array}
\end{equation}
is continuous, which is obviously the case since these maps are sum of a finite number of (scalar multiples of) projections. Therefore a linear form $\ell\colon\mathbb{K}^X\rightarrow(\mathbb{K},\tau)$ is continuous with respect to the weak topology if, and only if, it is continuous with respect to the product topology, and therefore is an element of $\mathbb{K}^{(X)}$ (according to Theorem~\ref{mainstatement}).
\end{proof}

%\begin{note}
%A direct and shorter proof of this result may be  obtained by applying Corollary~\ref{cor2}. 
%\end{note}

\bibliographystyle{abbrvnat}
% use the following instead if you encounter problems 
%\bibliographystyle{alpha}
\bibliography{poinsot-bib}

@book{BouTop,
title={Elements of mathematics - General topology},
author={Nicolas Bourbaki},
edition={2nd},
publisher={Springer},
year="2007"
}

@book{BouAlg,
title={Elements of mathematics - Algebra},
author={Nicolas Bourbaki},
edition={2nd},
publisher={Springer},
year="2007"
}

@book{Clifford,
title={The algebraic theory of semigroups, volume I},
author={Alfred H. Clifford, and G B. Preston},
series={Mathematical Survey},
number={7},
year={1961},
editor={American Mathematical Society}
}

@book{Cooke,
title={Infinite matrices and sequence spaces},
author={Richard G. Cooke},
publisher={Dover Publications, inc.},
year="1955"
}

@article{Dieudonne,
title={Linearly compact vector spaces and double vector spaces over sfields},
author={Jean Dieudonn\'e},
pages={13-19},
volume={73},
journal={Amer. J. Math.},
year="1951"
}

@book{Eisenbud,
title={Commutative algebra with a view toward algebraic geometry},
author={David Eisenbud},
publisher={Springer},
series={Graduate texts in mathematics},
volume={150},
year="1995"
}

@article{Kiltinen,
title={On the number of field topologies on an infinite field},
author={John O. Kiltinen},
journal={Proceedings of the American Mathematical Society},
volume={40},
number={1},
year="1973",
pages={30-36}
}

@book{Kothe,
title={Topological vector spaces I},
author={Gottfried K\"othe},
publisher={Springer-Verlag},
series={Die Grundlehren der mathematischen Wissenschaften},
volume={159},
year="1966"
}

@book{Lefschetz,
title={Algebraic topology},
author={Solomon Lefschetz},
series={Amerc. Math. Soc. Colloq. Pub.},
volume={27},
year="1942"
}

@article{Podewski,
title={The number of field topologies on countable fields},
author={Klaus-Peter Podewski},
volume={39},
number={1},
journal={Proceedings of the Amercian Mathematical Society},
pages={33-38},
year="1973"
}

@article{Poinsot,
title={M\"obius inversion formula for monoids with zero},
author={Laurent Poinsot, and G\'erard H. E. Duchamp, and Christophe Tollu},
volume={81},
number={3},
year={2010},
journal={Semigroup Forum},
pages={446-460}
}

@book{Treves,
title={Topological vector spaces, distributions and kernels},
author={François Tr\`eves},
volume={25},
series={Pure and applied mathematics},
editor={Academic Press}, 
year="1967"
}

@book{Warner,
title={Topological rings},
author={Seth Warner},
publisher={Elsevier},
series={North-Holland mathematics studies},
volume=178,
year="1993"
}
\label{sec:biblio}

\end{document}